\documentclass{amsart}
\usepackage{amsmath,amsthm,amssymb,amsfonts}
\input xy
\xyoption{all}
\usepackage[active]{srcltx}
\usepackage{tikz}
\usepackage{multirow}

\theoremstyle{plain}
\newtheorem{thm}{Theorem}[section]
\newtheorem{prop}[thm]{Proposition}
\newtheorem{lem}[thm]{Lemma}
\newtheorem{cor}[thm]{Corollary}

\theoremstyle{definition}
\newtheorem{defn}[thm]{Definition}

\numberwithin{table}{section}

\theoremstyle{remark}

\newcommand{\Hom}{\ensuremath{\mathrm{Hom}} }

\newcommand{\Aut}{\ensuremath{\mathrm{Aut}} }
\newcommand{\Inn}{\ensuremath{\mathrm{Inn}} }
\newcommand{\Out}{\ensuremath{\mathrm{Out}} }

\newcommand{\sA}{\ensuremath{\mathcal{A}} }
\newcommand{\sC}{\ensuremath{\mathcal{C}} }
\newcommand{\F}{\ensuremath{\mathcal{F}} }
\newcommand{\E}{\ensuremath{\mathcal{E}} }

\def\C{\ensuremath{\mathrm{C}} }
\def\D{\ensuremath{\mathrm{D}} }
\def\H{\ensuremath{\mathrm{H}} }
\def\N{\ensuremath{\mathrm{N}} }
\def\Z{\ensuremath{\mathrm{Z}} }

\newcommand{\normal}{\ensuremath{\unlhd}}
\newcommand{\gen}[1]{\langle #1\rangle}

\newcommand{\PSL}{\mathrm{PSL}}
\newcommand{\PSU}{\mathrm{PSU}}
\newcommand{\SL}{\mathrm{SL}}
\newcommand{\GL}{\mathrm{GL}}
\newcommand{\PGL}{\mathrm{PGL}}
\renewcommand{\O}{\mathrm{O}}
\newcommand{\Suz}{\mathrm{Suz}}
\newcommand{\mc}[1]{\mathcal{#1}}
\newcommand{\Modul}{\mathrm{Mod}}
\newcommand{\GA}{X}
\newcommand{\GB}{Y}

\newcommand{\rk}{\mathrm{rk}}

\newcommand*\circled[1]{%
  \tikz[baseline=(C.base)]\node[draw,circle,inner sep=0.5pt](C) {$#1$};\!
}

\newcommand{\darud}[1]{\ar @{--}@/^1pc/[#1]\ar @{--}@/_1pc/[#1]}
\newcommand{\lbar}[1]{\overline{#1}}

\begin{document}

\bibliographystyle{plain}

\title{Fusion systems on small $p$-groups}
\author{David A.\ Craven}
\author{Adam Glesser}
\date{\today}
\begin{abstract}
In this article we study several classes of `small' $2$-groups: we complete the classification, started in \cite{stancu2006}, of all saturated fusion systems on metacyclic $p$-groups for all primes $p$. We consider Suzuki $2$-groups, and classify all center-free saturated fusion systems on $2$-groups of $2$-rank $2$. We end by classifying all possible $\F$-centric, $\F$-radical subgroups in saturated fusion systems on $2$-groups of $2$-rank $2$.
%A metacyclic $p$-group is a $p$-group with a normal cyclic subgroup with cyclic quotient. Stancu proved in \cite[Proposition 5.4]{stancu2006} that every saturated fusion system on a metacyclic $p$-group with $p$ odd is resistant, i.e., there are no essential subgroups. In this note, we complete the classification by considering metacyclic $2$-groups. Namely, we show that trivial fusion system is the only saturated fusion system on a metacyclic $2$-group unless it is dihedral, semidihedral or generalized quaternion. We also provide an alternate proof of Stancu's result that avoids many of the technicalities in the original proof.
\end{abstract}
\maketitle

\section{Introduction}

In the classification of finite simple groups, a broad division is made between simple groups that are `small' in some sense and those that are `large'; one measure of size, less crude than the order, is the \emph{$p$-rank}: the rank of a largest elementary abelian subgroup of a Sylow $p$-subgroup of a group. In this article we will examine the possible fusion on `small' $p$-groups.

The modern way to describe fusion in finite groups is via saturated fusion systems (see, for example, \cite{blo2003} for background). Remarkably, there are saturated fusion systems that do not arise from finite groups and these \emph{exotic} systems are an interesting and rare phenomenon, particularly for the prime $2$.

One way of hunting for exotic fusion systems is to fix a finite $p$-group and ask for all possible saturated fusion systems on that $p$-group. If you can show that one of these does not come from a finite group, you've found an exotic system. Our work begins by completing the classification of saturated fusion systems on metacyclic $p$-groups, started by Stancu in \cite{stancu2006} with his description of saturated fusion systems on metacyclic $p$-groups for odd primes $p$. If $G$ is a finite group with Sylow $p$-subgroup $P$, then we denote the fusion system of $G$ by $\F_P(G)$.

\begin{thm}\label{t:intrometacyclic} Let $P$ be a metacyclic $2$-group.
\begin{enumerate}
\item If $P$ is dihedral or generalized quaternion, then there are three saturated fusion systems on $P$.
\item If $P$ is semidihedral then there are four saturated fusion systems on $P$.
\item If $P$ is the abelian group $C_{2^n}\times C_{2^n}$ then there are two saturated fusion systems on $P$.
\item If $P$ is none of the above groups, then there is a unique saturated fusion system on $P$, namely the fusion system $\F_P(P)$.
\end{enumerate}
\end{thm}

In Section \ref{sec:centerfree} we construct all of these saturated fusion systems as fusion systems of finite groups. For odd primes $p$, any saturated fusion system of a metacyclic $p$-group $P$ is the fusion system of a group of the form $P\rtimes A$, where $A$ is a $p'$-group of automorphisms of $P$. If every saturated fusion system on a $p$-group $P$ is of this form, then $P$ is called \emph{resistant}. We provide a new proof of Stancu's result that metacyclic $p$-groups, for $p$ odd, are resistant.

After a short section proving that \emph{Suzuki $2$-groups} -- nonabelian $2$-groups with an automorphism that permutes the involutions transitively -- are resistant, we turn to fusion systems on groups of $2$-rank $2$. It is shown in \cite{bcglo2007} that, up to isomorphism, one may reconstruct $\F$ from $\F/\Z(\F)$ and $P$, so it makes sense to study center-free fusion systems. In Section \ref{sec:centerfree}, we prove the following result, adapting the fusion-theoretic argument of Alperin in \cite{alperinbrauergorenstein1973}.

\begin{thm}\label{t:introcenterfree} Let $\F$ be a saturated fusion system on a finite $2$-group $P$ and suppose that $P$ has $2$-rank $2$. If $\Z(\F)=1$, then $P$ is dihedral, semidihedral, wreathed, $C_{2^n}\times C_{2^n}$ or the Sylow $2$-subgroup of $\PSU_3(4)$.
\end{thm}

This list is very similar to that of the Sylow $2$-subgroups of the finite simple groups of $2$-rank $2$, although we get more abelian $2$-groups.

Having determined the center-free fusion systems on $2$-groups of $2$-rank $2$, we consider all saturated fusion systems. In \cite{drv2007}, D\'iaz, Ruiz and Viruel gave a description of all saturated fusion systems on $p$-groups of $p$-rank $2$ for $p$ odd. The case of $2$-rank $2$ is not attempted in \cite{drv2007} and there are special difficulties in this case arising from the fact that there are many $2$-groups of $2$-rank $2$.

Here we will start the classification of such saturated fusion systems by exhibiting a classification of all possible $\F$-centric, $\F$-radical subgroups of such a fusion system, that is, classifying all $2$-groups of $2$-rank $2$ that possess a nontrivial automorphism of odd order (the automorphism groups of such groups, essentially, determine the fusion system). This does not appear to have been carried out in the literature, although such situations have been studied in, for example, \cite{thomas1989a} and \cite{thomas1989}. 

In Section \ref{sec:oddorderautos} we define $2$-groups $\GA_n$ and $\GB_n$, each of order $2^n$ for $n\geq 6$, which possess three involutions and an outer automorphism of order $3$. The next theorem comes from that section. (Our notation for various $p$-groups is defined at the end of the introduction.)

\begin{thm}\label{t:introoddauto} Let $P$ be a finite $2$-group of $2$-rank $2$. If $\Aut(P)$ is not a $2$-group, then $P$ is one of the following groups:
\begin{enumerate}
\item $C_{2^n}\times C_{2^n}$ or $Q_8\times C_{2^n}$ for some $n\geq 1$;
\item $Q_8\ast C_{2^n}$ (the central product of $Q_8$ and $C_{2^n}$) for some $n\geq 2$;
\item $Q_8\ast D_{2^n}$ or $Q_8\times Q_{2^n}$ for some $n\geq 3$;
\item $\GA_n$ or $\GB_n$ for $n\geq 6$ (noting that $\GA_6\cong \GB_6$);
\item $Q_8\wr C_2$;
\item the Sylow $2$-subgroup of $\PSU_3(4)$.
\end{enumerate}
\end{thm}

For each of these it is possible to analyze the possible fusion systems containing it. If $\F$ is a saturated fusion system on a finite $2$-group $P$, and if $U$ denotes an $\F$-essential subgroup of $P$, then by necessity $U$ is one of the $2$-groups above (or possibly $Q_8$); if $U$ is not normal in $P$ then the structure of $\N_P(U)$ and $\N_\F(U)$ is very restricted. This mirrors the classification in \cite{drv2007}, where there are many groups for which an essential subgroup has index $p$, but very few with larger index.

\bigskip

The organization of this paper is as follows: in the next section are the preliminary results that we need, with the following section proving Theorem \ref{t:intrometacyclic}, and considering metacyclic $p$-groups for $p$ odd. In Section \ref{sec:suzukigroups} we introduce Suzuki $2$-groups and prove that they are resistant, and in Section \ref{sec:centerfree} we prove Theorem \ref{t:introcenterfree}. The final section proves Theorem \ref{t:introoddauto}.

Our notation and definitions are largely standard now, and are taken, for example, from  \cite{craven2010}. Note that, as in \cite{craven2010}, our maps will be written on the right, and composition is left-to-right. The cyclic group of order $q$ is denoted by $C_q$, the dihedral, semidihedral, and generalized quaternion groups of order $2^n$ are denoted by $D_{2^n}$, $SD_{2^n}$ and $Q_{2^n}$ respectively, and the Klein $4$-group $C_2\times C_2$ is denoted by $V_4$. The wreathed $2$-group is defined to be $C_{2^n}\wr C_2$.

\section{Preliminaries}

We begin by finding sufficient conditions for the maximal subgroups of a finite $p$-group $P$ to be isomorphic via automorphisms of $P$. This will serve two important roles in the paper. First, for many of the groups in this paper, the structure of the maximal subgroups is understood, allowing us to verify immediately whether the sufficient conditions hold. Second, in this situation, any proper characteristic subgroup of $P$ is contained in $\Phi(P)$, the Frattini subgroup of $P$. This, for instance, allows us to use the famous result of Burnside (see \cite[Theorem 5.1.14]{gorenstein}) that the kernel of the natural map from $\Aut(P)$ to $\Aut(P/\Phi(P))$ is a $p$-group; this same result holds, of course, for any characteristic subgroup of $P$ contained in $\Phi(P)$.

\begin{prop}\label{p:maxsubs}
Let $P$ be a $2$-generator finite $p$-group and let $A$ be a subgroup of $\Out(P)$. If $|A|$ is divisible by $p$ and $\O_p(A) = 1$, then the preimage of $A$ in $\Aut(P)$ acts transitively on the set of maximal subgroups of $P$.
\end{prop}

\begin{proof} Let $\tilde A$ denote the full preimage of $A$ under the map $\Aut(P)\to \Out(P)$, and notice that $\Inn(P)$ is the kernel of the map $\tilde A\to \Aut(P/\Phi(P))\cong \GL_2(p)$. Thus $A$ is isomorphic to a subgroup of $G=\GL_2(p)$. We claim that (identifying $A$ with its image in $G$) $A$ contains $\SL_2(p)$. Since $\O_p(A)=1$, the number of Sylow $p$-subgroups of $A$ must be $p+1$ (as there are that many in $G$), and so $|A|\geq p(p+1)$; from the structure of $\GL_2(p)$, we see that the same holds for $A\cap \SL_2(p)$. In particular, the index of $A\cap \SL_2(p)$ in $\SL_2(p)$ is at most $p-1$, which immediately leads to a contradiction if $p\geq 5$ as $\SL_2(p)$ is quasisimple and $p\mid |\SL_2(p)|$. If $p=2$, then $A\cap \SL_2(p)=S_3$, and it is easy to see that $\SL_2(3)$ possesses no subgroup of index $2$. Therefore $A\cap \SL_2(p)=\SL_2(p)$.

Since $\SL_2(p)$ acts transitively on the projective line, we see that $A$ permutes transitively the $p+1$ different $1$-dimensional subspaces of the $2$-dimensional $\mathbb{F}_p$-vector space $P/\Phi(P)$; therefore $\tilde A$ must permute the maximal subgroups of $P$ transitively. In particular, the maximal subgroups are isomorphic.
\end{proof}

The conditions of this proposition were set up in such a way that if $Q$ is an $\F$-essential subgroup of a saturated fusion system $\F$, and $Q$ is $2$-generator, then all maximal subgroups of $Q$ are isomorphic. Notice that all we really needed in the proof was that $A$ contained elements of order $p+1$ in $\SL_2(p)$, since they permute the $1$-dimensional subspaces transitively. In the case where $p=2$, we can therefore say something stronger.

\begin{cor}\label{c:maximal}
Let $P$ be a $2$-generator $2$-group. If $\Aut(P)$ is not a $2$-group, then $\Aut(P)$ acts transitively on the maximal subgroups of $P$.
\end{cor}

\begin{proof}
Since $\GL_2(2)\cong S_3$, $\Out(P)$ contains an element of order $p+1=3$ which, by the above remark, transitively permutes the three subspaces of dimension $1$.
\end{proof}

In the next corollary, we collect several consequences of Proposition \ref{p:maxsubs} and Corollary \ref{c:maximal}. Recall from \cite[Theorem 5.4.5]{gorenstein} that the $2$-groups with maximal class are the dihedral, semidihedral and generalized quaternion groups (all of which are $2$-generator groups).

\begin{cor}\label{c:omnibus}
Let $P$ be a finite $p$-group.
\begin{enumerate}
\item If $\Aut(P)$ acts transitively on the maximal subgroups of $P$, then every proper characteristic subgroup $C$ of $P$ is contained in $\Phi(P)$. In particular, there is a bijection between the maximal subgroups of $P$ and $P/C$ and the kernel of the natural map $\Aut(P) \to \Aut(P/C)$ is a $p$-group.\label{cp:propchar}
\item Let $p$ be odd. If $P \cong C_{p^m} \times C_{p^n}$ and there exists a subgroup $A$ of $\Aut(P)$ such that $|A|$ is divisible by $p$ and $\O_p(A) = 1$, then $m = n$.\label{cp:2genabelodd}
\end{enumerate}
Assume that $p=2$.
\begin{enumerate}
\setcounter{enumi}{2}
\item If $P$ possesses a self-centralizing subgroup isomorphic to $V_4$, then $P$ is dihedral or semidihedral.\label{cp:selfcentv4}
\end{enumerate}
Assume further that $\Aut(P)$ is not a $2$-group.
\begin{enumerate}
\setcounter{enumi}{3}
\item If $P \cong C_{2^m}\times C_{2^n}$, then $m = n$.\label{cp:2genabeleven}
\item If $P$ has maximal class, then $P\cong Q_8$.\label{cp:maxclass}
\item If $P$ has a cyclic maximal subgroup, then $P\cong V_4$ or $P \cong Q_8$.\label{cp:cyclicmaximal}
\end{enumerate}
\end{cor}

\begin{proof}
\begin{enumerate}
\item Let $C$ be a proper characteristic subgroup of $P$. As it is contained in some maximal subgroup and is fixed by all automorphisms of $P$, it is contained in all maximal subgroups and hence is contained in $\Phi(P)$. This implies the bijection between maximal subgroups and the final statement is an immediate consequence of Burnside's result, \cite[Theorem 5.1.14]{gorenstein}.
\item As $C_{p^m} \times C_{p^n}$ has maximal subgroups $C_{p^m} \times C_{p^{n-1}}$ and $C_{p^{m-1}} \times C_{p^n}$, the result follows from Proposition \ref{p:maxsubs}.
\item By \cite[Satz III.14.23]{huppert}, $P$ has maximal class. Furthermore, $P$ cannot be a generalized quaternion group as it contains more than one involution.
\item The proof is the same as in (\ref{cp:2genabelodd}) but utilizing Corollary \ref{c:maximal}.
\item Any maximal-class $2$-group except $Q_8$ has a pair of nonisomorphic maximal subgroups and so the result follows from Corollary \ref{c:maximal}.
\item Any such $P$ is a $2$-generator group. If $P$ is abelian, then by Corollary \ref{c:maximal} and part (\ref{cp:2genabeleven}), $P\cong V_4$. If $P$ is nonabelian, then by \cite[Theorem 5.4.4]{gorenstein}, $P$ has maximal class or is a modular $2$-group. The latter case is impossible since every modular $2$-group has at least two nonisomorphic maximal subgroups (see Section \ref{sec:modular2groups} for the definition of modular $2$-groups and this statement). The result now follows from (\ref{cp:maxclass}).
\end{enumerate}
\end{proof}

In the case where the homocyclic abelian subgroup has exponent greater than $2$, we can also pin down the structure of the ambient $2$-group.

\begin{lem}\label{l:homocyclicwreathed} 
Let $\F$ be a saturated fusion system on a finite $2$-group $P$ and let $Q$ be a proper, abelian, $\F$-centric, $\F$-radical subgroup of $P$.
\begin{enumerate}
\item If $|Q|\leq 4$, then $Q\cong V_4$ and $P$ is dihedral or semidihedral.
\item If $|Q|>4$, and $P$ has $2$-rank $2$, then $|P:Q|=2$ and $P$ is wreathed.
\end{enumerate}
\end{lem}

\begin{proof} 
Since $Q$ is $\F$-centric and $\F$-radical, $\Aut(Q)$ is not a $2$-group. In particular, $Q$ is not cyclic. Therefore, by Corollary \ref{c:omnibus}.\ref{cp:selfcentv4}, if $Q \cong V_4$, then $P$ is dihedral or semidihedral. So, we may assume that $|Q| > 4$. By Corollary \ref{c:omnibus}.\ref{cp:2genabeleven}, $Q\cong C_{2^n}\times C_{2^n}$. Let $R=\N_P(Q) > Q$. Since $Q$ is $\F$-radical, $\Aut_\F(Q)\cong S_3$ and $|R:Q|=2$. Let $x$ and $y$ be generators of $Q$, and let $t$ be an element of $R\setminus Q$; considering the action of $\Aut_\F(Q)$ on $Z\setminus\{1\}$ (on which it must act faithfully), we see that $x^t\notin \gen x$, and so we may choose $y=x^t$. To prove that $R$ is a wreathed $2$-group, we need to show that $R\setminus\gen{x,y}$ contains an element of order $2$.

Since $t^2\in Q$, and $t$ centralizes $t^2$, write $t^2=(xy)^i$. We have
\[ (x^{-i}t)^2=x^{-i}t^2(t^{-1}x^{-i}t)=x^{-i}x^iy^iy^{-i}=1,\]
so that $R$ is wreathed. Since $Q$ is the unique maximal abelian subgroup of $R$, it is characteristic in $R$ and hence $Q\normal \N_P(R)$. Thus $P=R$, as required. 
\end{proof}

\section{Metacyclic $2$-Groups}

\subsection{Modular $2$-groups}
\label{sec:modular2groups}

Wong \cite{wong1964} proved that any finite group with a modular $2$-group as a Sylow $2$-subgroup is $2$-nilpotent. Recently, Ensslen and K\"ulshammer \cite[Theorem 7]{ensslenkulshammer2008} proved a version of Wong's results for blocks, namely, that a block with defect group a modular $2$-group is a nilpotent block.

Here we show that there is a single saturated fusion system on a modular $2$-group $P$, namely $\F_P(P)$; this generalizes the above two results. Recall that the modular $2$-groups have the form 
\[ \Modul_n=\gen{x,y \mid x^{2^{n-1}} = 1 = y^2,\; yxy = x^{2^{n-2}+1}}, \]
where $n \geq 4$. In particular, modular $2$-groups are $2$-generator $2$-groups with $2$-rank $2$. The subgroup $\gen x$ is a cyclic subgroup of index $2$, and the subgroup $\gen{x^2,y}$ is a noncyclic abelian subgroup of index $2$, proving the assertion in the proof of Corollary \ref{c:omnibus}.\ref{cp:cyclicmaximal}. 
%(Since $\gen{x^{2^{n-2}},y}$ is a noncyclic group of order $4$, there must be a noncyclic maximal subgroup.) 
As a warm-up to the general theorem on metacyclic $2$-groups, we prove the following, now straightforward, proposition. 

\begin{prop}\label{p:modular}
The only saturated fusion system on a modular $2$-group is the trivial fusion system.
\end{prop}

\begin{proof}
Let $P = \Modul_n$ and let $\F$ be a saturated fusion system on $P$. As $P$ has nonisomorphic maximal subgroups, Corollary \ref{c:maximal} implies that $\Aut(P)$ is a $2$-group. Moreover, as every proper subgroup of $P$ is abelian and $P$ is not dihedral, semidihedral or wreathed, Lemma \ref{l:homocyclicwreathed} implies that $P$ has no $\F$-centric, $\F$-radical subgroups. By Alperin's fusion theorem and the Sylow axiom for fusion systems, $\F$ is trivial.
\end{proof}

%\begin{proof} As we noted earlier, not all maximal subgroups of $P=\Modul_n$ are isomorphic; Corollary \ref{c:maximal} now implies that $\Aut(P)$ is a $2$-group. Consequently, it suffices to show there are no $\F$-centric, $\F$-radical subgroups of $P$; let $Q$ be such a subgroup. Using the notation for the modular $2$-groups and its generators introduced above, the three maximal subgroups, $\gen x$, $\gen{x^2,y}$ and $\gen{x^2,xy}$, are all abelian, so $Q$ must be abelian; this contradicts Lemma \ref{l:homocyclicwreathed}, so $Q$ cannot exist. Hence the only saturated fusion system on $P$ is $\F_P(P)$, as claimed.
%\end{proof}

\subsection{Automorphisms of metacyclic subgroups}

We need to understand how $p'$-automorphisms act on metacyclic $p$-groups. For $p=2$ we obtain a complete description of which metacyclic $p$-groups afford $p'$-automorphisms, and for odd $p$ we determine which metacyclic $p$-groups have an automorphism group consistent with being a centric radical subgroup in a larger $p$-group.

\begin{prop}\label{p:metacyclic}
Let $Q$ be a metacyclic $p$-group such that $\Aut(Q)$ is not a $p$-group.
\begin{enumerate}
\item If $p=2$, then $Q\cong C_{2^n}\times C_{2^n}$ or $Q \cong Q_8$.\label{pp:nabel-meta}
\item If $p$ is odd and $A$ is a subgroup of $\Out(Q)$ such that $|A|$ is divisible by $p$ and $\O_p(A)=1$, then $Q \cong C_{p^n} \times C_{p^n}$.\label{pp:odd-meta-abel}
\end{enumerate}
\end{prop}
\begin{proof}
If $Q$ is abelian then the result holds by Corollary \ref{c:omnibus}.\ref{cp:2genabeleven}, so we assume that $Q$ is nonabelian. As $Q'$ is cyclic, the unique involution of $Q'$ generates a characteristic subgroup $Z$ of $Q$ of order $2$. Since $Z \leq \Phi(Q)$, we see that $\Aut(Q/Z)$ is not a $2$-group, so by induction $Q/Z\cong Q_8$ or $Q/Z$ is homocyclic abelian.

In the former case, $|Q/Q'| = |(Q/Z)/(Q'/Z)| = |Q_8/Q'_8| = 4$ and so, by \cite[Theorem 5.4.5]{gorenstein}, $Q$ is dihedral, semidihedral or generalized quaternion. As none of these has $Q_8$ as a proper quotient by the center, we derive a contradiction. If $Q/Z$ is homocyclic abelian, then the preimage in $Q$ of the subgroup generated by all involutions in $Q/Z$ is a characteristic subgroup $C$ of $Q$ of order $8$ containing all involutions of $Q$. If $\Aut(C)$ is a $2$-group, then an automorphism of odd order in $\Aut(Q)$ induces the trivial automorphism on $C$, so in particular it induces the trivial automorphism on $\Omega_1(Q/Z)$, contradicting \cite[Theorem 5.2.4]{gorenstein}, which states that a $p'$-automorphism of an abelian $p$-group cannot act trivially on the elements of order $p$. Hence $\Aut(C)$ is not a $2$-group, and so $C\cong Q_8$. This implies that $Q$ has a unique involution; in particular, any abelian subgroup of $Q$ is cyclic and so, by \cite[Theorem 5.4.10(ii)]{gorenstein}, $Q$ is a generalized quaternion group. Corollary \ref{c:omnibus}.\ref{cp:maxclass} implies that $Q\cong Q_8$, proving (\ref{pp:nabel-meta}).

Suppose that $Q$ is nonabelian. Since $Q'$ is cyclic, $Z=\Omega_1(Q')$ has order $p$; the hypotheses are inherited by $Q/Z$, so we may assume by induction that $Q/Z$ is homocyclic abelian, of order $p^{2n}$. It is well known (see for example \cite[Lemma 1.2]{mazza2003}) that $Q$ possesses a unique elementary abelian subgroup $R$ of order $p^2$, and since this is contained in one maximal subgroup, it is contained within all of them, as they are all isomorphic by Proposition \ref{p:maxsubs}. Again, the hypotheses are inherited by $Q/R$ since $R\leq \Phi(P)$, and so $Q/R$ is homocyclic abelian, of order $p^{2m}$. However, now $|Q|=p^{2n+1}=p^{2m+2}$, a contradiction. Hence, $Q$ is abelian, and \ref{pp:odd-meta-abel} follows from Corollary \ref{c:omnibus}.\ref{cp:2genabelodd}.
\end{proof}

In \cite{linckelmannmazza2009}, Linckelmann and Mazza define the Dade group, $\D(\F,P)$, of a saturated fusion system $\F$ on a finite $p$-group $P$ as the $\F$-stable elements in the Dade group, $\D(P)$, of $P$. Furthermore, they give a criterion for when $\D(\F,P) = \D(P)$. 

\begin{thm}[{{\cite[Theorem 6.1]{linckelmannmazza2009}}}]\label{t:linckelmannmazza}
Let $\F$ be saturated fusion system on a finite $2$-group $P$. If every $\F$-essential subgroup of $P$ is isomorphic to $V_4$ or $Q_8$, then $\D(\F,P)=\D(P)$.
\end{thm}  

This is used to prove that $\D(\F,P)=\D(P)$ whenever $p=2$ and $P$ is of maximal class. We will use our preparatory work to generalize their result to all metacyclic $2$-groups.

\begin{cor}\label{c:dade}
Let $\F$ be a saturated fusion system on a finite $2$-group $P$. If $P$ is metacyclic, then $\D(\F,P)=\D(P)$.
\end{cor}

\begin{proof}
As $P$ is metacyclic, it is not wreathed and so, by Lemma \ref{l:homocyclicwreathed}, any abelian $\F$-essential subgroup of $P$ is isomorphic to $V_4$. Furthermore, by Proposition \ref{p:metacyclic}.\ref{pp:nabel-meta}, every nonabelian $\F$-essential subgroup is isomorphic to $Q_8$. The result now follows from Theorem \ref{t:linckelmannmazza}.  
\end{proof}

%\begin{proof}
%As $P$ is metacyclic, it is not wreathed and so if $P$ has an abelian $\F$-essential subgroup $Q$, Lemma \ref{l:homocyclicwreathed} implies that $Q\cong V_4$ and hence $\D(\F,P) = \D(P)$. By Proposition \ref{p:metacyclic}.\ref{pp:nabel-meta}, we may assume that every $\F$-essential subgroup of $P$ is isomorphic to $Q_8$ which again implies that $\D(\F,P)=\D(P)$.
%\end{proof}

\subsection{Fusion systems on metacyclic $p$-groups}

If $\F$ is a saturated fusion system on a finite $p$-group $P$ and $Z$ is a central subgroup of $\F$, i.e., $\F = \C_\F(Z)$, then the following result of Kessar and Linckelmann gives a handy way of passing between the centric radical subgroups of $\F$ and $\F/Z$. We will use it in the proof of Theorem \ref{t:metacyclic}.

\begin{prop}[{{\cite[Proposition 3.1]{kessarlinckelmann2008}}}]\label{p:KL}
Let $\F$ be a saturated fusion system on a finite $p$-group $P$ and let $Z$ be a subgroup of $P$ such that $\F = \C_\F(Z)$. If $Q$ is an $\F$-centric, $\F$-radical subgroup of $P$, then $Q/Z$ is an $\F/Z$-centric, $\F/Z$-radical subgroup of $P/Z$.
\end{prop}

\begin{lem}\label{l:Zstrclosed}
Let $\F$ be a saturated fusion system on a finite $p$-group $P$. If the proper $\F$-centric, $\F$-radical subgroups of $P$ have a common center $Z$, then $Z$ is strongly $\F$-closed.
\end{lem}
\begin{proof}
Let $T \leq Z$, let $Q$ be a proper $\F$-centric, $\F$-radical subgroup of $P$ containing $T$ and let $\varphi \in \Aut_\F(Q)$. As $T \leq Z=\Z(Q)$ and $\Z(Q)$ is characteristic in $Q$, we have $\varphi(T) \leq \varphi(\Z(Q)) = \Z(Q) = Z$.
\end{proof}

The following theorem strengthens Proposition \ref{p:modular}, showing that the only metacyclic $2$-groups that give rise to nontrivial saturated fusion systems are homocyclic, dihedral, semidihedral or generalized quaternion. At the end of Section \ref{sec:centerfree}, we will construct all saturated fusion systems on these groups. Benjamin Sambale independently proved this result using different methods in \cite[Theorem 1]{sambale2009un}.

\begin{thm}\label{t:metacyclic}
Let $\F$ be a saturated fusion system on a finite metacyclic $2$-group $P$. If $P$ is not homocyclic abelian, dihedral, semidihedral or generalized quaternion, then $\F$ is trivial.
\end{thm}
\begin{proof}
Let $Q$ be an $\F$-essential subgroup of $P$. If $Q$ is abelian, then Lemma \ref{l:homocyclicwreathed} implies that $P$ is dihedral, semidihedral, or wreathed; since wreathed $2$-groups are not metacyclic, this case cannot occur.

Therefore, every proper $\F$-essential subgroup of $P$ is nonabelian and hence, by Proposition \ref{p:metacyclic}.\ref{pp:nabel-meta}, isomorphic to $Q_8$. As each $\F$-essential subgroup $Q$ is $\F$-centric, $1 \neq \Z(P) \leq \Z(Q)$. As the last group has order 2, $\Z(Q) = \Z(P)$. By Lemma \ref{l:Zstrclosed}, $Z = \Z(P)$ is strongly $\F$-closed. The unique nontrivial element of $\Z(P)$ is thus central in $\F$. We conclude that $Z = \Z(\F)$. By Proposition \ref{p:KL}, $Q/Z$ is $\F/Z$-centric and so $P/Z$ has a self-centralizing subgroup isomorphic to $V_4$. In particular, $P$ has maximal class, a contradiction. 

The only remaining case is where $P$ has no proper $\F$-essential subgroups. In this case, we invoke Proposition \ref{p:metacyclic} to conclude that $\Aut(P)$ is a $2$-group. This proves that $\F$ is trivial.
\end{proof}

Blocks of finite groups with dihedral, semidihedral or generalized quaternion defect groups are tame, and so this theorem yields the following immediate corollary.

\begin{cor}
Any wild $2$-block with a metacyclic, nonhomocyclic defect group is nilpotent.
\end{cor}

Another application of our theorem to $2$-blocks with metacyclic defect groups is Linckelmann's ``gluing problem'' (see \cite[Conjecture 4.2]{linckelmann2004}) for blocks. Recently, Sejong Park showed \cite[Theorem 1.2]{park2010} that the gluing problem has a unique solution for tame blocks. Theorem \ref{t:metacyclic} enables us to extend his result to all blocks with metacyclic defect groups. To state this result, we need to set up some notation.

For a saturated fusion system $\F$ on a finite $p$-group $P$, let $\sC$ denote the poset of $\F$-conjugacy classes of chains of $\F$-centric subgroups of $P$ (the ordering is induced by taking subchains). For $i \in \mathbb{N}$, let $\sA^i_\F$ denote the covariant functor sending a chain $\sigma \in \sC$ to the abelian group $\H^i(\Aut_\F(\sigma), k^{\times})$. This allows us to define the cohomology of $\sC$ with coefficients in $\sA^i_\F$. 

K\"ulshammer and Puig showed in \cite{kulshammerpuig1990} that if $\F$ is the fusion system of a block of a finite group, then the block determines an element of $\H^0(\sC, \sA^2_\F)$. Linckelmann conjectured that this element was in the image of the map
\[
\H^2(\F^c, k^{\times}) \to \H^0(\sC, \sA^2_\F)
\]
and showed (see \cite{linckelmann2010un}) that every saturated fusion system $\F$ gives rise to the exact sequence,
\[
0 \to \H^1(\sC, \sA^1_\F) \to H^2(\F^c, k^{\times}) \to \H^0(\sC, \sA^2_\F) \to \H^2(\sC, \sA^1_\F) \to \H^3(\F^c, k^{\times}),
\]
which we will use without reference in the proof of the next corollary. 

\begin{cor}
Let $\F$ be a saturated fusion system on a finite $2$-group $P$ and let $k$ be an algebraically closed field of characteristic $2$. If $P$ is metacyclic, then $\H^2(\F^c, k^{\times}) = \H^0(\sC, \sA^2_\F) = 0$. In particular, the gluing problem for $2$-blocks with metacyclic defect group has a unique solution.
\end{cor}

\begin{proof}
If $\sigma$ is a chain of $\F$-centric subgroups in $\sC$ and if $\H^2(\Aut_\F(\sigma), k^{\times}) = 0$, then $\sA^2_\F = 0$ and hence $\H^0(\sC, \sA^2_\F) = 0$. This implies that $\H^1(\sC, \sA^1_\F) \cong \H^2(\F^c, k^{\times})$ and it would then remain to show that $\H^1(\sC, \sA^1_\F) = 0$. By \cite[Theorem 1.2]{park2010} and Theorem \ref{t:metacyclic}, it will suffice to consider the cases where either $\F = \F_P(P)$ or $P \cong C_{2^n} \times C_{2^n}$. In the former case $\sA^1_\F = \sA^2_\F = 0$ which implies that $\H^1(\sC, \sA^1_\F) \cong \H^2(\F^c, k^{\times}) = 0$ giving the result.

If $P \cong C_{2^n} \times C_{2^n}$, then there is only one $\F$-centric subgroup, namely $P$, and hence $\H^1(\sC, \sA^1_F) = 0$. Furthermore, by Theorem \ref{t:tamefusion}, if $\F \neq \F_P(P)$, then $\F = \F_P(P \rtimes C_3)$ and hence $\Aut_\F(P) \cong C_3$. As, $\H^2(C_3, k^{\times}) = 0$, the proof is complete.
\end{proof}

We now deal with the odd primes case, giving a structural proof of Stancu's result from \cite{stancu2006}.

\begin{thm}\label{t:metacyclicodd}
Let $\F$ be a saturated fusion system on a finite $p$-group $P$ with $p$ odd. If $P$ is metacyclic, then $\F=\N_\F(P)$.
\end{thm}
\begin{proof}
Let $P$ be a minimal counterexample and let $\F$ be a saturated fusion system on $P$ such that $\N_\F(P) \neq \F$. If $P$ is abelian, the extension axiom implies every automorphism of a fully $\F$-normalized subgroup extends to $P$ and hence $\F=\N_\F(P)$, so $P$ is nonabelian. By Alperin's fusion theorem and Proposition \ref{p:metacyclic}.\ref{pp:odd-meta-abel}, $P$ contains an $\F$-essential subgroup $Q \cong C_{p^n} \times C_{p^n}$ for some $n \in \mathbb{N}$. In fact, by \cite[Lemma 2.1]{mazza2003}, $Q$ is the unique homocyclic abelian subgroup of $P$ of order $p^{2n}$ and hence $Q$ is characteristic in $P$. Since $Q$ is fully $\F$-normalized, $\Aut_P(Q)$ is a Sylow $p$-subgroup of $\Aut_{\F}(Q)$. As this latter group is isomorphic to a subgroup of $\mathrm{GL}_2(p)$, 
\[ |P:Q|=|\N_P(Q):\C_P(Q)|=|\Aut_P(Q)|=p,\]
and hence $|P| = p^{2n+1}$. 

If $n=1$, then $P$ is an extraspecial group of order $p^3$. As $Q\cong C_p\times C_p$ contains exactly $p^2-1$ elements of order $p$, $P$ must have exponent $p^2$. As $Q$ is $\F$-essential, the proof of Proposition \ref{p:maxsubs} implies that $\Aut_\F(Q)=\Out_\F(Q)$  contains $\SL_2(p)$ and hence contains the automorphism $\phi$ that inverts all elements of $Q$. Since $\phi$ centralizes $\Aut_P(Q)$, the extension axiom gives an extension $\psi\in\Aut_\F(P)$ of $\phi$, which may be chosen to have order $2$. Since $P$ has $p$ cyclic subgroups of index $p$, $\psi$ must fix (at least) one of them; call it $R$. The center of $P$ is contained in $R$ and $\psi$ inverts $\Z(P)$. Thus, $\psi$ induces the unique automorphism of order $2$ on $R$ (which is, of course, inversion). Therefore, $\psi$ inverts all elements of $QR = P$, a contradiction since $P$ in nonabelian. 

We conclude that $n>1$. As $P/\Phi(Q)$ is metacyclic, the minimality of $P$ implies that $P/\Phi(Q)$ is resistant. Furthermore, since $Q$ is the unique $\F$-essential subgroup of $P$ and since $\Phi(Q)$ is characteristic in $P$, it follows that $\Phi(Q)$ is strongly $\F$-closed. Thus, $\lbar{\F} = \F/\Phi(Q)$ is a saturated fusion system on $\lbar{P} = P/\Phi(Q)$. 

Since $\lbar{P}$ is resistant, it is the only $\lbar{\F}$-centric, $\lbar{\F}$-radical subgroup. In particular, $\lbar{Q}$ is not $\lbar{\F}$-centric and $\lbar{\F}$-radical. However, since $\lbar{Q}$ is elementary abelian, the kernel of the canonical surjection $\Aut_\F(Q) \to \Aut_{\lbar{\F}}(\lbar{Q})$ is a normal $p$-subgroup. It is, therefore, trivial since $Q$ is $\F$-radical. Thus, the above surjection is an isomorphism and $\lbar{Q}$ is $\lbar{F}$-radical. It follows that $\lbar{Q}$ cannot be $\lbar{F}$-centric. But, as $[\lbar{P}:\lbar{Q}] = p$, this implies $\C_{\lbar{P}}(\lbar{Q}) = \lbar{P}$, i.e., $\lbar{Q}$ is contained in the center of $\lbar{P}$. Therefore, $\lbar{P}$ is abelian. This is impossible, however, as then $\Aut_{\lbar{P}}(\lbar{Q})$ is trivial while $p \mid |\Aut_{\lbar{\F}}(\lbar{Q})|$, contradicting the saturation of $\lbar{\F}$.

% As $P$ has $p$-rank 2, the unique elementary abelian subgroup $R$ of $P$ of rank 2 contains all of the elements of order $p$ and hence is strongly $\F$-closed. By \cite[Lemma 2.4.4]{glesser2009un}, there exists $\phi \in \Aut_\F(Q)$ such that $N_\phi = Q$. If $\bar{\phi}$ denotes the induced map in $\Aut_{\F/R}(Q/R)$, then $N_{\bar{\phi}} = Q/R$, i.e., $\bar{\phi}$ does not extend to $P/R$ and hence $\N_{\F/R}(P/R)\neq \F/R$, contradicting the minimality of $P$.
\end{proof}

\section{Suzuki $2$-groups}
\label{sec:suzukigroups}

In this short section we discuss Suzuki $2$-groups and show that they are always resistant. This result is needed in the following section.

\begin{defn} A finite $2$-group is a \emph{Suzuki} $2$-group if $P$ is nonabelian, possesses more than one involution, and has an odd-order automorphism $\phi$ such that $\phi$ transitively permutes the elements of $\Omega_1(P)$.
\end{defn}

Examples of Suzuki $2$-groups are the Sylow $2$-subgroups of the Suzuki simple groups $^2B_2(q)$, as well as the Sylow $2$-subgroups of the groups $\PSL_3(2^n)$ and $\PSU_3(2^n)$. The main theorem on these groups is a result of Graham Higman.

\begin{thm}[Higman \cite{Higman1963}]\label{t:Higman} If $P$ is a Suzuki $2$-group, then
\[ \Omega_1(P)=\Z(P)=\Phi(P)=P',\]
and so $P$ has exponent $4$ and class $2$.
\end{thm}

In order to prove our result, we use the following result of Aschbacher.

\begin{prop}[Aschbacher]\label{p:strongcentral} Let $\F$ be a saturated fusion system on a finite $p$-group $P$. If $Q$ is a strongly $\F$-closed subgroup of $P$, then $\F=\N_\F(Q)$ if and only if $Q$ possesses a central series of strongly $\F$-closed subgroups 
\[1=Q_0\leq Q_1\leq \cdots \leq Q_n=Q.\]
\end{prop}
\begin{proof} See \cite[Proposition 3.4]{craven2010}.\end{proof}

\begin{thm}\label{p:suzukiresistant} Suzuki $2$-groups are resistant.
\end{thm}

\begin{proof}
Let $P$ be a Suzuki $2$-group. Since $Q=\Omega_1(P)$ is elementary abelian and contains all involutions, it is strongly $\F$-closed. Therefore, the sequence $1\leq Q\leq P$ is a central series of strongly $\F$-closed subgroups. Hence, by Proposition \ref{p:strongcentral}, $\F=\N_\F(P)$, giving us the result.
\end{proof}

\section{Center-free and simple saturated fusion systems of $2$-rank $2$}
\label{sec:centerfree}

Using the results of the previous sections, we classify center-free saturated fusion systems on $2$-groups of $2$-rank $2$. In \cite{bcglo2007}, it was shown that the quotient of a saturated fusion system $\F$ by its center preserves structure, in the sense that one may reconstruct $\F$, up to isomorphism, from $\F/\Z(\F)$ and the original $p$-group on which $\F$ is defined. Iterating this procedure results in a saturated fusion system with trivial center. Thus, understanding center-free saturated fusion systems is of considerable importance. For instance, Corollary 6.14 in \cite{bcglo2007} says that $\F$ is an exotic fusion system if and only if $\F/\Z(\F)$ is exotic. At the end of the section, we classify all simple saturated fusion systems on $2$-groups of $2$-rank $2$. 

Let $P$ be a finite $2$-group of $2$-rank $2$ and let $\F$ be a center-free
saturated fusion system on $P$. Since $\Z(\F) = 1$, there is some $\F$-centric, $\F$-radical subgroup $U$ and an automorphism $\phi$ such that $\Omega_1(\Z(U)) \not\leq \C_U(\phi)$. In particular, $\Omega_1(\Z(U))$ has order $4$ and therefore contains all involutions in $U$.

If $U$ is abelian, then Lemma \ref{l:homocyclicwreathed} implies that $P$ is dihedral, semidihedral, wreathed or homocyclic abelian. If $U$ is nonabelian then $U$ is a Suzuki $2$-group, and by \cite[Lemma 2.2.6]{glsvol6}, $U$ is the Sylow $2$-subgroup of $\PSU_3(4)$, a Suzuki $2$-group of order $2^6$. Suppose that $U < P$; since $U$ is $\F$-centric, the fusion system $\N_\F(U)$ is constrained; if $G$ is a finite group whose fusion system is that of $\N_\F(U)$, then $\N_P(U)$ has $2$-rank $3$ by \cite[Lemma 5.2.9(c)]{glsvol6}. Since $P$ has $2$-rank $2$, we conclude that $P=U,$ giving the following theorem.

\begin{thm}\label{t:centerfreeclass} Let $\F$ be a saturated fusion system on a finite $2$-group $P$ with $2$-rank $2$. If $\Z(\F)=1$, then $P$ is dihedral, semidihedral, wreathed, homocyclic abelian, or the Sylow $2$-subgroup of $\PSU_3(4)$.
\end{thm}

It is easy to find a center-free fusion system on each of these groups: the simple groups $\PSL_2(q)$, $\PSL_3(q)$, $\PSU_3(q)$ (for $q$ odd), and $\PSU_3(4)$ give examples for all but the homocyclic abelian, and in this case the group $(C_{2^n}\times C_{2^n})\rtimes C_3$ is an example. We next construct all saturated fusion systems on the groups mentioned in this result.

\begin{defn} Let $\F$ be a saturated fusion system on a finite $p$-group $P$. The \textit{essential rank} of $\F$, denoted by $\rk_e(\F)$, is the number of $\F$-conjugacy classes of $\F$-essential subgroups of $P$.  

\end{defn}

\begin{thm}\label{t:tamefusion}
Let $\F$ be a saturated fusion system on a finite $2$-group $P$. If $P$ is dihedral, semidihedral, generalized quaternion or wreathed, then $\rk_e(\F) \leq 2$ and $\F$ is isomorphic to the fusion system on $P$ of one of the groups in the following table:

%\begin{center}\begin{tabular}{|c|c|c|c|}
%\hline $\rk_e(\F)$ & $\mathrm{Dihedral}$  &  $\mathrm{Semidihedral}$ & $\mathrm{Quaternion}$ \\ \hline 
%0 &\multicolumn{3}{|c|}{$\F_P(P)$} \\ 
%\hline 1 & $\PGL_2(q)$  &  $\PSL_2(q^2).C_2$,  $\GL_2(q)$ & $\SL_2(q).C_2$ \\ 
%\hline 2 & $\PSL_2(q)$  &  $\PSL_3(q)$  & $\SL_2(q)$ \\ 
%\hline 
%\end{tabular}\end{center}

%\begin{center}\begin{tabular}{|c|c|c|c|c|}
%\hline $\rk_e(\F)$ & $D_{2^n}$  &  $SD_{2^n}$ & $Q_{2^n}$ & $C_{2^n} \wr C_2$  \\ \hline 
%$0$ &\multicolumn{4}{|c|}{$\F_P(P)$} \\ 
%\hline $1$ & $\PGL_2(q)$  &  $\PSL_2(q^2)\rtimes C_2$,  $\GL_2(q)$ & $\SL_2(q)\rtimes C_2$ & $(C_{2^n} \times %C_{2^n}) \rtimes S_3$, $\GL_2(q)$ \\ 
%\hline \multirow{2}{*}{$2$} & $\PSL_2(q)$  &  $\PSL_3(q)$  & $\SL_2(q)$ & $\PSL_3(q)$ \\ 
 %& $q \equiv \pm 1 \pmod 8$ & $q \equiv 3 \pmod 4$ & $q\  \mathrm{odd}$ & $q \equiv 1 \pmod 4$ \\
%\hline 
%\end{tabular}\end{center}

\begin{center}\begin{tabular}{|c|c|cc|c|}
\hline $\rk_e(\F)$ & $0$ & \multicolumn{2}{|c|}{$1$} & $2$ \\
\hline \multirow{2}{*}{$D_{2^n}$} & \multirow{8}{*}{$\F_P(P)$} & \multicolumn{2}{|c|}{\multirow{2}{*}{$\PGL_2(q)$}} & $\PSL_2(q)$
\\ & & & & $q\equiv \pm1\bmod 8$
\\ \cline{1-1}\cline{3-5} \multirow{2}{*}{$SD_{2^n}$} & & \multirow{2}{*}{$\PSL_2(q^2)\rtimes C_2$,} & $\GL_2(q)$ & $\PSL_3(q)$
\\ & & & $q\equiv 3\bmod 4$  & $q\equiv 3\bmod 4$ 
\\ \cline{1-1}\cline{3-5} \multirow{2}{*}{$Q_{2^n}$} & & \multicolumn{2}{|c|}{\multirow{2}{*}{$\SL_2(q).C_2$}} & $\SL_2(q)$
\\ & & & & $q$ odd
\\ \cline{1-1}\cline{3-5} \multirow{2}{*}{$\C_{2^n}\wr C_2$} & & \multirow{2}{*}{$(C_{2^n}\times C_{2^n})\rtimes S_3$,} & $\GL_2(q)$ & $\PSL_3(q)$
\\ & & & $q\equiv 1\bmod 4$ & $q\equiv 1\bmod 4$
\\ \hline
\end{tabular}
\end{center}

\noindent If $P$ is homocyclic abelian or the Sylow $2$-subgroup of $\PSU_3(4)$, then $\rk_e(\F) = 0$ and $\F$ is isomorphic to the fusion system on $P$ of one of the groups in the following table:

\begin{center}\begin{tabular}{|c|c|}
\hline $C_{2^n} \times C_{2^n}$ & $\mathrm{Suz}$ \\ \hline  $\F_P(P)$, $\F_P(P\rtimes C_3)$ & $\F_P(P)$, $\F_P(P\rtimes C_3)$, $\F_P(P\rtimes C_5)$, $\F_P(P\rtimes C_{15})$  \\ 
\hline
\end{tabular}\end{center}
\end{thm}

\begin{proof} The proof of this theorem is broken up into three parts: the dihedral, semidihedral and quaternion groups, wreathed groups, and the homocylic abelian and Suzuki groups.

\subsection*{Dihedral, Semidihedral and Quaternion $2$-groups}

The saturated fusion systems on $D_{2^n}$, $SD_{2^n}$ and $Q_{2^n}$ are quite
similar. We will try, as much as possible, to proceed in parallel for all three classes of groups. 

By Lemma \ref{l:homocyclicwreathed} and Proposition \ref{p:metacyclic}.\ref{pp:nabel-meta}, any $\F$-essential subgroup of $P$ is isomorphic to $V_4$ or $Q_8$ and has $\F$-outer automorphism group isomorphic to $S_3$. In any of the cases, $|\Z(P)|=2$ and $\Z(P)$ is contained in every $\F$-essential subgroup of $P$.  By Alperin's fusion theorem and the fact that the automorphism groups of dihedral, semidihedral and quaternion $2$-groups are again $2$-groups, $\F$ is completely determined by the automorphism groups of the $\F$-essential subgroups. If $\rk_e(\F)=0$, then $\F=\F_P(P)$ for any of our potential $P$, so we may assume that there is at least one $\F$-essential subgroup.

If $P$ is dihedral, then the $\F$-essential subgroups must all be isomorphic to $V_4$. The dihedral groups have two conjugacy classes of $V_4$-subgroups (that are $\Out(P)$-conjugate) and three conjugacy classes of involutions. An automorphism of order $3$ of $V_4$ permutes its three involutions, and as one of these involutions is central in $P$, two of the conjugacy classes of involutions in $P$ will fuse in $\F$. If one or both classes of $V_4$ subgroups are essential then there are two or a single $\F$-conjugacy class of involutions respectively. In the former case this is the fusion system of $\PGL_2(q)$ (for the appropriate odd $q$) and in the latter case this is the fusion system of $\PSL_2(q)$ (again, for the appropriate odd $q$).

%If $P$ is generalized quaternion, then it has a unique involution and so any $\F$-essential subgroup is isomorphic to $Q_8$. As, $P/\Z(P)$ is dihedral, there are two conjugacy classes of subgroups isomorphic to $Q_8$. If both classes yield an $\F$-essential subgroup, then $\F$ is the fusion system of $\SL_2(q)$, while if there is only one if comes from an extension of $\SL_2(q)$ by $C_2$.

If $P$ is generalized quaternion, then every $\F$-essential subgroup is isomorphic to $Q_8$. As $P$ has a unique involution,  $\Z(\F)=\Z(P)\neq 1$. Thus, $\F/\Z(\F)$ is a fusion system on a dihedral group and so by \cite[Corollary 6.14]{bcglo2007}, $\F$ is the fusion system of a finite group. In fact, it is a central extension of a fusion system of a dihedral group.

Finally, if $P$ is semidihedral, then there are three maximal subgroups: a dihedral subgroup, $D$, a cyclic subgroup, $C$, and a generalized quaternion subgroup, $Q$. The two $D$-conjugacy classes of subgroups isomorphic with $V_4$ are fused in $P$, as are the two $Q$-conjugacy classes of subgroups isomorphic to $Q_8$. (If $|P|=16$, there is a unique subgroup of $P$ isomorphic to $Q_8$.) Therefore there is a single $P$-conjugacy class of subgroups isomorphic with $V_4$, and similarly for subgroups isomorphic with $Q_8$.
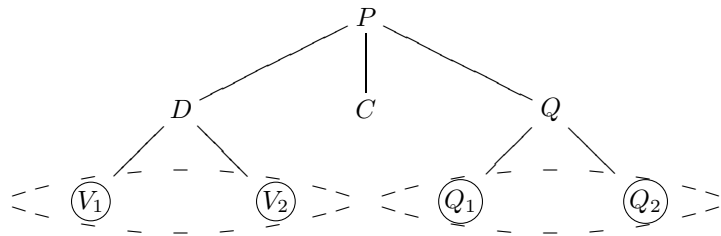
\[
\xymatrix@! @=1pc{
&& & & P \ar@{-}[dll] \ar@{-}[drr] \ar@{-}[d] & & & & \\
&& D \ar@{-}[dl] \ar@{-}[dr] & &  C & & Q \ar@{-}[dl] \ar@{-}[dr] & & \\
\darud{rrrr}&\circled{V_1} & & \circled{V_2} & \darud{rrrr} & \circled{Q_1} & & \circled{Q_2} &
} 
\]
We conclude that there are at most two $\F$-conjugacy classes of $\F$-essential subgroups of $P$, one isomorphic to $V_4$ and one isomorphic to $Q_8$. As there are two conjugacy classes of involutions in a semidihedral group, there is an $\F$-essential isomorphic to $V_4$ if and only if the involutions form a single $\F$-conjugacy class. Semidihedral groups also have precisely two conjugacy classes of elements of order $4$ (and these are permuted by an automorphism of order $3$ of $Q_8$). Thus, there is an $\F$-essential subgroup isomorphic to $Q_8$ if and only if the elements of order $4$ form a single $\F$-conjugacy class. If both of these situations occur, our fusion system comes from $\PSL_3(q)$. If only the latter situation occurs, we get $\GL_2(q)$. Finally, if only the former situation occurs, then we get the fusion system coming from a nonsplit extension of $\PSL_2(q^2)$ by the field automorphism. (This group has a unique conjugacy class of involutions but must have more than one class of elements of order $4$ since some lie outside of $\PSL_2(q^2)$.)

\subsection*{Wreathed $2$-groups}

Let $P$ be the wreathed group $C_{2^n}\wr C_2$ for some $n\geq 2$ (as $n=1$ gives $P\cong D_8$), and let $Q$ denote the base group $C_{2^n}\times C_{2^n}$ generated by $a$ and $b$, a characteristic subgroup of $P$ with $|P:Q|=2$. Since $P$ is $2$-generator, and one, but not all maximal subgroups of $P$ are abelian (e.g., $C_{2^{n-1}}\wr C_2$ is a subgroup of $P$), $\Aut(P)$ is a $2$-group by Corollary \ref{c:maximal}. Hence it suffices, as in the previous case, to determine all possible essential subgroups of $P$. Note that $P$ has a cyclic center of order $2^n$, generated by $ab$ and lying inside $Q$. 

Fix an element $t$ in $P\setminus Q$ with $a^t=b$ and let $S$ be a subgroup of $P$ with $\C_P(S)\leq S$ such that $\Aut(S)$ not a $2$-group. In \cite[Lemma 1.3]{alperinbrauergorenstein1970}, it is stated that $S$ is either $Q$ or $P$-conjugate to the central product $R$ of $\Z(P)$ and $\gen{(ab^{-1})^{2^{n-2}},t}\cong Q_8$. Since no proof is offered there we sketch a proof here. If $S$ is abelian then $S$  is homocyclic abelian and contains $\Z(P)$, so that $S=Q$. Hence, $S$ is nonabelian, and contains $\Z(P)$. The quotient $P/\Z(P)$ is dihedral of order $2^{n+1}$, and since any automorphism of odd order acts trivially on $\Z(P)$, $\Aut(S/\Z(P))$ cannot be a $2$-group. Hence $S/\Z(P)\cong V_4$, and there are two conjugacy classes of such subgroups; write $R$ and $T$ for preimages in $P$ of representatives of these subgroups. They are generated by the central element of $P/\Z(P)$, namely $\Z(P)a^{2^{n-1}}$, and either $\Z(P)t$ or $\Z(P)at$, and so the two subgroups of interest are
\[ R=\gen{ab,a^{2^{n-1}},t}\text{ and }T=\gen{ab,a^{2^{n-1}},at}=\gen{a^{2^{n-1}},at}.\]
As the cyclic subgroup $\gen{at}$ and the noncyclic subgroup $\gen{ab,a^{2^{n-1}}}$ are both maximal in $T$, there is no automorphism of $T$ of order $3$. This leaves $R$, which is a central product of $\Z(P)$ and a $Q_8$ subgroup, as claimed.

There are therefore four saturated fusion systems on $P$, depending on whether one of the subgroups $Q$, $R$, both $Q$ and $R$, or neither $Q$ nor $R$, are essential subgroups. These are the fusion systems of (respectively): $(C_{2^n}\times C_{2^n})\rtimes S_3$ (with $S_3$ acting by permutations on $a$, $b$ and $a^{-1}b^{-1}$ where $a$ and $b$ are as above); the centralizer of an involution in $\PSL_3(q)$ for $q\equiv 1\bmod 4$, i.e., $\GL_2(q)$;  $\PSL_3(q)$ for $q\equiv 1\bmod 4$ (and $\PSU_3(q)$ for $q\equiv 3\bmod 4$); and $P$ itself.

\subsection*{Resistant Cases}

The other groups mentioned above are the homocyclic abelian groups $P_n=C_{2^n}\times C_{2^n}$, and the Sylow $2$-subgroup of $\PSU_3(4)$, which we will denote by $\Suz$. These $2$-groups are resistant (in the first case since all abelian $p$-groups are resistant and in the second case by Theorem \ref{p:suzukiresistant}), and so to determine all saturated fusion systems on these groups it suffices to determine $\Out(P_n)$ and $\Out(\Suz)$.

In the first case, $\Out(P_n)$ has order $2^\alpha\cdot 3$ for some $\alpha$ by Proposition \ref{p:maxsubs}, and so there are two saturated fusion systems on $P_n$, namely those of $P_n$ and $P_n\rtimes C_3$.

In the second case, it is easy to see (see the following section) that $\Out(\Suz)$ has automorphisms of order $3$ and $5$, and (up to raising by powers) there are unique such automorphisms. Hence there are four saturated fusion systems on $\Suz$, namely those of the groups $\Suz$, $\Suz\rtimes C_3$, $\Suz\rtimes C_5$ and $\Suz\rtimes C_{15}$.
\end{proof}

We finish the section by finding all simple fusion systems on $2$-groups of $2$-rank $2$. Recall that if $\F$ is a saturated fusion system on a finite $p$-group $P$ and if $\E$ is a saturated subsystem of $\F$ on a strongly $\F$-closed subgroup $Q$ of $P$, then $\E$ is \textit{weakly normal} in $\F$ if for all subgroups $S \leq T \leq Q$ and for all $\phi \in \Hom_\F(T,Q)$, conjugation by $\phi$ induces a bijection $\Hom_\E(S,T) \to \Hom_\E(S\phi, T\phi)$. A saturated fusion system is \textit{simple} if its only weakly normal subsystems are the fusion system of the trivial group and itself. Note that while this definition of simplicity differs, \textit{a priori}, from that given by Aschbacher in \cite{aschbacher2008}, recent work of the first author \cite{craven2010un2} shows that these definitions are, in fact, equivalent. 

\begin{thm}\label{t:simplefusion}
Let $\F$ be a saturated fusion system on a finite $2$-group $P$ with $2$-rank $2$. If $\F$ is simple, then $P$ is dihedral, semidihedral or wreathed, and $\F$ is either the fusion system of $\PSL_2(q)$ for $q \equiv \pm 1 \bmod 8$ or $\PSL_3(q)$ for $q$ odd.
\end{thm}

\begin{proof}
By \cite[Proposition 6.2]{stancu2006}, $\O_p(\F) = 1$. Therefore, $\Z(\F) = 1$ and we may apply Theorem \ref{t:centerfreeclass}, implying that $P$ is dihedral, semidihedral, wreathed, homocyclic abelian or the Sylow $2$-subgroup of $\PSU_3(4)$. However, $P$ is not homocyclic abelian since the only simple fusion system on an abelian $p$-group is $\F_{C_p}(C_p)$ and $P$ is not the Sylow $2$-subgroup of $\PSU_3(4)$ since in this case, by Theorem \ref{p:suzukiresistant}, $P$ is normal in $\F$. The result now follows from Theorem \ref{t:tamefusion} and the fact that a simple fusion system coming from a finite group must come from a simple group. 
\end{proof}

\section{Odd-order automorphisms of $2$-groups}
\label{sec:oddorderautos}

In this section we will determine all $2$-groups of $2$-rank $2$ possessing a nontrivial odd-order automorphism, the first step in classifying all saturated fusion systems on $2$-groups of $2$-rank $2$, which would complete the project in \cite{drv2007}.

In order to state the main theorem of this section, we introduce the following notation for several classes of $2$-groups. Let $Q_{n,m}=Q_{2^n}\times Q_{2^m}$, $QC_{n,m}=Q_{2^n}\times C_{2^m}$ and $QD_{n,m} = Q_{2^n} \times D_{2^n}$. The notation $QC_{n,m}^*$ and $QD_{n,m}^*$ are defined similarly but as central products instead of direct products. The first two of these series have exactly three involutions and, if $n =3$, each of these possesses an automorphism of order $3$ acting in the obvious way on the $Q_8$-factor and trivially on the other factor; the group $Q_{3,3}$ possesses another automorphism of order $3$ acting nontrivially on both factors. We also write $C_{n,m}=C_{2^n}\times C_{2^m}$. The Sylow $2$-subgroup of $\PSU_3(4)$ will again make an appearance as the only Suzuki $2$-group with three involutions, and here we will again refer to it as $\Suz$.

There are two more infinite series of groups to define; both are (nonsplit) extensions of $Q_8\times C_{2^n}$ by $C_2$. Let $Q_8$ be generated by $a$ and $b$, and let the cyclic factor be generated by $c$. The groups $\GA_n$ and $\GB_n$ (of order $2^n$) have a normal subgroup $QC_{3,n-4}$ generated by $a$, $b$ and $c$, with $d$ outside having order $4$ and commuting with $a$ and $b$, and moving $c$. In $\GB_n$, the element $d$ inverts $c$, and in $\GA_n$ it sends $c$ to $c^{-1}a^2$. In $\GA_n$, $d^2$ is the element $z$ of order $2$ in $\gen c$, and in $\GB_n$ we have $d^2=a^2z$.

More formally, we have
\begin{align*} \GA_n=&\langle\; a,b,c,d\mid a^4=c^{2^{n-4}}=1,\;b^2=a^2,\;a^{-1}ba=b^{-1},
\\ &d^2=c^{2^{n-5}},\;d^{-1}cd=c^{-1}a^2,\;[a,c]=[b,c]=[a,d]=[b,d]=1\;\rangle,\text{ and}
\\ \GB_n=&\langle\;a,b,c,d\mid a^4=c^{2^{n-4}}=1,\;b^2=a^2,\;a^{-1}ba=b^{-1},
\\ &d^2=a^2c^{2^{n-5}},\;d^{-1}cd=c^{-1},\;[a,c]=[b,c]=[a,d]=[b,d]=1\;\rangle.\end{align*}
Notice that $\GA_6=\GB_6$. The automorphism of order $3$ on these two groups acts in the obvious way on $\gen{a,b} \cong Q_8$ and trivially on $c$ and $d$.

\begin{thm}\label{t:oddorderauts} Let $P$ be a finite $2$-group. If $P$ has exactly three involutions, then $P$ possesses a nontrivial odd-order automorphism $\phi$ if and only if $P$ is one of the following:
\begin{enumerate}
\item $QC_{3,n}$ for some $n\geq 1$ with $o(\phi)=3$;
\item $Q_{3,n}$ for some $n\geq 3$ with $o(\phi)=3$ (there are two different automorphisms if $n=3$);
\item $C_{n,n}$ for some $n\geq 1$ with $o(\phi)=3$;
\item $\GA_n$ for some $n\geq 7$ with $o(\phi)=3$;
\item $\GB_n$ for some $n\geq 6$ with $o(\phi)=3$;
\item the Suzuki $2$-group $\Suz$ with $o(\phi)=3$ or $o(\phi)=5$.
\end{enumerate}
\end{thm}

\begin{thm}\label{t:oddorderauts2} Let $P$ be a finite $2$-group of $2$-rank $2$. If $P$ has more than three involutions, then $P$ possesses an odd-order automorphism if and only if $P$ is one of the following:
\begin{enumerate}
\item $Q_8\ast C_{2^n}$ for some $n\geq 2$ with $o(\phi)=3$;
\item $Q_8\ast D_{2^n}$ for some $n\geq 3$ with $o(\phi)=3$, or $n=3$ and $o(\phi)=5$;
\item $Q_8\wr C_2$ with $o(\phi)=3$.
\end{enumerate}
\end{thm}

The proofs of Theorems \ref{t:oddorderauts} and \ref{t:oddorderauts2} require a few results from the literature on $2$-groups and their automorphisms. We firstly need MacWilliams's four-generator theorem.

\begin{thm}[MacWilliams \cite{macwilliams1970}] Let $P$ be a finite $2$-group, and suppose that $P$ has no normal, elementary abelian subgroup of order $8$. Then $P$ has sectional $2$-rank at most $4$.
\end{thm}

Consequently, if $P$ has $2$-rank $2$ then $P$ is at most $4$-generator. Our next preliminary result is a lemma of Mazurov, although it appears in various guises elsewhere, such as \cite[Theorem 1.3]{hawkes1973}.

\begin{lem}[{{\cite[Lemma 2]{mazurov1969}}}]\label{l:fpfaction} Let $P$ be a finite $2$-group. If $P$ possesses an odd-order automorphism that fixes all involutions of $P$ and acts fixed-point freely on $P/\Phi(P)$, then $P$ is a special $2$-group and $\Omega_1(P)=\Phi(P)$, i.e., $P'=\Z(P)=\Phi(P) = \Omega_1(P)$.
\end{lem}

Note that when we say a group has $n$ involutions, we mean that it has \emph{exactly} $n$ involutions.

\begin{thm}[Janko {{\cite[Theorem 2.2]{janko2005}}}]\label{t:jankoCPWmetacyclic} Let $P$ be a finite $2$-group with three involutions. If $P$ contains a normal subgroup $W$ isomorphic with $C_4\times C_4$, then $\C_P(W)$ is metacyclic and $\Omega_2(\C_P(W))=W$.
\end{thm}

The next result is well known, and we include a short proof because we need slightly more information than usual, namely that the three involutions in the maximal subgroup are all central.

\begin{lem}\label{l:2rank2threeinvos} If $P$ is a $2$-group of $2$-rank $2$, then there is a subgroup $Q$ of index at most $2$ such that $Q$ has three involutions, all central in $Q$, or $P$ is dihedral or semidihedral, in which case there is a cyclic maximal subgroup $Q$.
\end{lem}
\begin{proof} By \cite[Theorem 5.4.10]{gorenstein}, if $P$ is not dihedral or semidihedral then $P$ contains a noncyclic normal abelian subgroup, $Y$. As $Y$ is noncyclic, it must have $2$-rank $2$ and since it is normal and abelian, $X = \Omega_1(Y) \cong V_4$ is normal in $P$. Let $Q = \C_P(X)$; as $Q$ commutes with $X$, the three involutions in $X$ are the only involutions in $Q$. Finally, $[P:Q] = [P:\C_P(X)] = |\Aut_P(X)| \leq 2$.
\end{proof}

The following trivial lemma will be useful at several points.

\begin{lem}\label{l:threeorseven} Let $Q$ be a normal subgroup of a $2$-group $P$ such that $P/Q\cong V_4$ and such that each maximal subgroup of $P$ containing $Q$ possesses three involutions. If $Q$ has a single involution, then $P$ has seven involutions. Otherwise, both $Q$ and $P$ have three involutions.
\end{lem}
\begin{proof} Every element of $P$ lies in one of its maximal subgroups, and so if $Q$ has a single involution, each of the maximal subgroups contains two involutions not in $Q$ and hence $P$ has seven involutions. If $Q$ has three involutions, then these are all the involutions in each maximal subgroup and hence in $P$.  
\end{proof}

We will denote by $\mc P$ the set of all pairs $(P,\phi)$, where $P$ is a $2$-group from Theorems \ref{t:oddorderauts} and \ref{t:oddorderauts2}, and $\phi$ is a nontrivial odd-order automorphism of $P$. It is easy to determine the action of $\phi$ on a $2$-group $P$: in all cases apart from $C_{n,n}$, $Q_{3,3}$, $Q_8\wr C_2$, $QD_{3,3}^*$ with $o(\phi)=5$, and the group $\Suz$, the group is a central product of a $2$-group and the group $Q_8$. Here, the automorphism has order $3$, and acts on $Q_8$ in the obvious way and acts trivially on the other factor of the central product. If $P$ is abelian then the action is obvious, and if $P\cong Q_{3,3}$ then we get the action just described, and also an automorphism acting diagonally on both factors; this automorphism lifts to an automorphism of $Q_8\wr C_2$, which describes this case. The final cases are $QD_{3,3}^*$ and $\Suz$, which are two small groups and hence amenable to computation via Magma.

Using a computer, we verify Theorems \ref{t:oddorderauts} and \ref{t:oddorderauts2} for groups of order up to $2^9$. This will avoid any complications due to the orders of groups being small.

\begin{prop}\label{p:smallorder} Let $P$ be a finite $2$-group of order at most $2^9$ and suppose that $P$ has $2$-rank $2$. If $\Aut(P)$ is not a $2$-group, then $P$ is on the lists given in Theorems \ref{t:oddorderauts} and \ref{t:oddorderauts2}.
\end{prop}
\begin{proof} This is an easy verification using the computer algebra package Magma. The exact entries in the SmallGroup database are as in Table \ref{smallgroups1}.
\begin{table}
\begin{center}\begin{tabular}{cc|c|}
\multicolumn{3}{c}{\multirow{2}{*}{\textbf{Three involutions}}}
\\
\\ \hline \multicolumn{1}{|c|}{Order} & SmallGroup & Isomorphism Types
\\ \hline  \multicolumn{1}{|c|}{16} & 2, 12 & $C_{2,2}$, $QC_{3,1}$
\\  \multicolumn{1}{|c|}{32} & 26 & $QC_{3,2}$
\\  \multicolumn{1}{|c|}{64} & 2, 126, 238, 239, 245 & $C_{8,8}$, $QC_{3,3}$, $\GB_6$, $Q_{3,3}$, $\Suz$
\\  \multicolumn{1}{|c|}{128} & 914, 2095, 2114, 2127 & $QC_{3,4}$, $\GA_7$, $Q_{3,4}$, $\GB_7$
\\  \multicolumn{1}{|c|}{256} & 39, 6647, 26918, 26937, 26950 & $C_{4,4}$, $QC_{3,5}$, $\GA_8$, $Q_{3,5}$, $\GB_8$
\\  \multicolumn{1}{|c|}{512} & 60795, 420460, 420479, 420492& $QC_{3,6}$, $\GA_9$, $Q_{3,6}$, $\GB_9$
\\ \hline
\\ \multicolumn{3}{c}{\multirow{2}{*}{\textbf{More than three involutions}}}
\\
\\ \hline  \multicolumn{1}{|c|}{Order} & SmallGroup & Isomorphism Types
\\ \hline  \multicolumn{1}{|c|}{16} & 13  & $QC_{3,2}^*$
\\  \multicolumn{1}{|c|}{32} &  38, 50 & $QC_{3,3}^*$, $QD_{3,3}^*$
\\  \multicolumn{1}{|c|}{64} & 185, 259 & $QC_{3,4}^*$, $QD_{3,4}^*$
\\  \multicolumn{1}{|c|}{128} & 937, 990, 2149 & $Q_8\wr C_2$, $QC_{3,5}^*$, $QD_{3,5}^*$
\\  \multicolumn{1}{|c|}{256} & 6725, 26972 & $QC_{3,6}^*$, $QD_{3,6}^*$
\\  \multicolumn{1}{|c|}{512} & 60897, 420514& $QC_{3,7}^*$, $QD_{3,7}^*$
\\ \hline
\end{tabular}\end{center}
\caption{Table of small groups}\label{smallgroups1}
\end{table}
\end{proof}

We use the following elementary, but useful, result in the proofs of Proposition \ref{p:nophiinvmaxs}, Lemma \ref{l:no57} and Proposition \ref{p:4gen3genphiinv}.

\begin{prop}\label{p:4dimvec} Let $V$ be a $4$-dimensional vector space over $\mathbb{F}_2$ and let $g\in\GL(V)$. 
\begin{enumerate}
\item If $o(g)=3$ then either $g$ acts fixed point freely, or $g$ fixes three points that lie in a subspace of dimension $2$ (i.e., they
are not linearly independent).
\item If $o(g)=5$ then $g$ acts fixed point freely.
\end{enumerate}
\end{prop}

\begin{proof}
Notice that $\GL_4(2)$, the set of all linear transformations of $V$, is isomorphic with $A_8$. Hence there are two conjugacy classes of
elements of order $3$ and a single conjugacy class of elements of order $5$. Let $v_1$, $v_2$, $v_3$ and $v_4$ form a basis for $V$. If $g$ has order $3$ then there are two possibilities: either $g$ acts by sending $v_1\mapsto v_2$, $v_2\mapsto v_3$, $v_3\mapsto v_1$ and
fixing $v_4$,  or by sending $v_i\mapsto v_{i+1}$ and $v_{i+1}\mapsto v_i+v_{i+1}$, for $i \in \{1,3\}$. The latter action has no fixed
points while the former has three: $v_4$, $v_1+v_2+v_3$ and $v_1+v_2+v_3+v_4$; these form the three nonzero elements of a
$2$-dimensional subspace. If $g$ has order $5$, then $g$ (is conjugate to a map that) sends $v_i$ to $v_{i+1}$ for
$i \in \{1,2,3\}$ and sends $v_4$ to $v_1+v_2+v_3+v_4$; this map has no fixed points.
\end{proof}

The remainder of the proof is structured as follows. We first examine the case where $\phi$ fixes no maximal subgroup of $P$. Note that there is a $\phi$-invariant bijection between the subspaces of $P/\Phi(P)$ with dimension $k$ and those with codimension $k$ since $\phi$ is conjugate to an orthogonal transformation, which commutes with taking orthogonal complements. In particular, $\phi$ fixes no maximal subgroup of $P$ if and only if $\phi$ acts fixed-point freely on $P/\Phi(P)$. We use this several times in the next proof in order to employ Lemma \ref{l:fpfaction}. We then attack the situation where there is a $\phi$-invariant maximal subgroup of $P$ in four steps: reduce to the case where $P$ is $3$-generator or $4$-generator and $o(\phi) = 3$; prove the result when $P$ is $3$-generator; prove the result when $P$ is $4$-generator and has a $3$-generator $\phi$-invariant maximal subgroups; and, finally, prove that when $P$ is 4-generator, it does have a $3$-generator $\phi$-invariant maximal subgroup.

\begin{prop}\label{p:nophiinvmaxs} Let $P$ be a finite $2$-group with $2$-rank $2$ such that $\Aut(P)$ is not a $2$-group. If $\phi$ is a nontrivial odd-order automorphism of $P$ and no maximal subgroup of $P$ is $\phi$-invariant, then $P$ is on the lists given in Theorems \ref{t:oddorderauts} and \ref{t:oddorderauts2}. Specifically, $P$ is $2$-generator or $4$-generator and either 
\begin{enumerate}
\item $o(\phi)=3$ and $P$ is isomorphic to $C_{n,n}$ or $\Suz$, or
%\end{enumerate}
%\noindent or
%\begin{enumerate}\setcounter{enumi}{1}
\item $o(\phi)=5$ and $P$ is isomorphic to $\Suz$ or $QD_{3,3}^*$. 
\end{enumerate}
\end{prop}

\begin{proof} If $P$ is $2$-generator then $\phi$ has order $3$ and permutes the maximal subgroups, all of which must have three involutions by Lemma \ref{l:2rank2threeinvos}. By Lemma \ref{l:threeorseven}, $P$ has either three or seven involutions. If it has seven, then $\Phi(P)$ has a single involution, so that $P$ is generated by two involutions; hence $P$ is a dihedral group, a contradiction as then $\Aut(P)$ is a $2$-group (also, no dihedral $2$-group has seven involutions). Hence, $P$ has three involutions. If $P$ is abelian, then $P\cong C_{n,n}$ by Corollary \ref{c:omnibus}.\ref{cp:2genabeleven}. Otherwise, $P$ is nonabelian, and $\phi$ either acts transitively on the involutions, so that $P$ is a Suzuki $2$-group, or $\phi$ acts trivially on the involutions. In either case, $P$ has order 16 (by Theorem \ref{t:Higman} and Lemma \ref{l:fpfaction}, respectively) and so, by Proposition \ref{p:smallorder}, is isomorphic to $C_{2,2}$ or $QC_{3,1}$, a contradiction since $C_{2,2}$ is abelian and $QC_{3,1}$ has nonisomorphic maximal subgroups.

If $P$ is $3$-generator, then there are seven maximal subgroups. Consequently, if $o(\phi)=3$, then there is a $\phi$-invariant maximal subgroup. We conclude that $o(\phi) = 7$ and that $\phi$ transitively permutes the maximal subgroups of $P$. By Lemma \ref{l:2rank2threeinvos}, again, all have exactly three involutions. The automorphism $\phi$ also permutes the seven normal subgroups with quotient $V_4$ and so, by Lemma \ref{l:threeorseven}, $P$ has three or seven involutions. If $\phi$ permutes the involutions transitively, then there must be seven involutions and the involutions are all central by Lemma \ref{l:2rank2threeinvos}; this implies the $2$-rank of $P$ is greater than 2, a contradiction. We conclude that $\phi$ acts trivially on the involutions. However, Lemma \ref{l:fpfaction} then implies that $\Omega_1(P)=\Z(P)$. In particular, the involutions of $P$ are all central, contradicting the $2$-rank of $P$ if there are seven involutions. If there are only three involutions, then $|P| = 32$ and, by Proposition \ref{p:smallorder}, $P \cong QC_{3,2}$, a contradiction since $QC_{3,2}$ has no automorphism of order 7 (or, alternatively, since not all of its maximal subgroups are isomorphic).

% $\Omega_1(P)=\Phi(P)=\Z(P)$, so that $P$ has $2$-rank $3$ in the second case, and $|P|=2^5$ in the first case, so that $P$ does not exist by Proposition \ref{p:smallorder}.

Finally, we assume that $P$ is $4$-generator. As $P$ possesses no $\phi$-invariant maximal subgroup and there are fifteen maximal subgroups, we conclude that $o(\phi)=3$ or $o(\phi)=5$. By Lemma \ref{l:2rank2threeinvos}, $P$ possesses a maximal subgroup $Q$ with three, central involutions. Let $R=Q\cap Q\phi$; if $R$ has three involutions then they are centralized by both $Q$ and $Q\phi$, so that $|\Omega_1(\Z(P))|\geq 4$. As $P$ has $2$-rank $2$, this means that either $P$ has three involutions or $R$ has a single involution.

%Suppose firstly that $P$ has three involutions. If $\phi$ acts nontrivially on these three involutions (so that $o(\phi)=3$), then it permutes them transitively, so that $P$ is a Suzuki $2$-group. Hence, we assume that $\phi$ acts trivially on $\Omega_1(P)$, and since it also acts freely on $P/\Phi(P)$, by Lemma \ref{l:fpfaction} $P$ is a special $2$-group and $\Omega_1(P)=\Phi(P)$; thus $|P|= 2^6$, and by Proposition \ref{p:smallorder}, $(P,\phi)$ lies in the set $\mc P$, so we easily see that $P\cong Q_{3,3}$ and $o(\phi)=3$ or $P\cong \Suz$ and $o(\phi)=5$, as claimed.

Suppose first that $P$ has three involutions. On one hand, if $\phi$ acts nontrivially on these three involutions (so that $o(\phi)=3$), then it permutes them transitively, so that $P$ is a Suzuki $2$-group. By Theorem \ref{t:Higman}, $|P| = 2^6$ and so $P \cong \Suz$. On the other hand, if $\phi$ acts trivially on the involutions of $P$, then, by Lemma \ref{l:fpfaction}, $|P| = 2^6$ and so $(P, \phi) \in \mc P$ by Proposition \ref{p:smallorder}. One easily sees in this case that $P \cong Q_{3,3}$ with $o(\phi) = 3$ or that $P \cong \Suz$ with $o(\phi) = 5$.

We now assume that $R$ has a single involution. If $R$ is $\phi$-invariant, then $\phi$ induces a nontrivial automorphism on $R$ (as no subgroup of $P$ with index $8$ and containing $\Phi(P)$ is $\phi$-invariant). The only group with a single involution and with nontrivial automorphisms of odd order is $Q_8$, so that $P$ has order $2^5$, in contradiction with Proposition \ref{p:smallorder}; hence $R$ is not $\phi$-invariant. However, if $o(\phi)=3$, then Proposition \ref{p:4dimvec} makes it easy to see that the intersection of $Q$ and $Q\phi$ is $\phi$-invariant. Hence, $o(\phi)=5$.
% if $V=\gen{x_1,x_2,x_3,x_4}$ is a $4$-dimensional $\mathbb{F}_2$-vector space, then a linear map $f$ of order $3$ acts freely if and only if (with the right choice of the $x_i$) $x_1f=x_2$, $x_3f=x_4$, $x_2f=x_1+x_2$ and $x_4f=x_3+x_4$. If $Q$ is a maximal subgroup, corresponding with (say) $\gen{x_1,x_2,x_3}$, then $Q\phi$ corresponds with $\gen{x_1,x_2,x_4}$ and their intersection is readily seen to be $\phi$-invariant. 

As $P$ is $4$-generator, $R$ cannot be cyclic, so is generalized quaternion. Since $P$ is $4$-generator and $R$ is $2$-generator, $\Phi(P)=\Phi(R)\cong C_{2^n}$ is cyclic. If $n=1$ then $P$ has order $2^5$, and $P\cong QD_{3,3}^*$ by Proposition \ref{p:smallorder}. If $n\geq 2$, then let $X$ be the (characteristic) subgroup of index $4$ in $\Phi(P)$, and let $G=P/X$. The group $G$ has order $2^6$, has $4$-generators, has an automorphism of order $5$, and $\Phi(G)$ is cyclic. However, examining the groups of order $2^6$ using Magma, we find that there is no such $2$-group, and hence $P$ cannot exist if $n\geq 2$.
%The subgroup $R=Q\cap Q\phi$, which has index $4$, is a normal $\phi$-invariant subgroup, contained in the maximal subgroups $Q\phi^i$. Each of the subgroups $Q\phi^i$ possesses exactly three involutions, and $\phi$ acts nontrivially on $R$; if $R$ has a single involution then $R\cong Q_8$ and so $|P|=2^5$, and we can just check the list. Hence $R$ has three involutions, so by Lemma \ref{l:threeorseven} $P$ has exactly three involutions. If $\phi$ acts nontrivially on these three involutions, then it permutes them transitively, so that $P$ is a Suzuki $2$-group. Hence $P$ acts trivially on $\Omega_1(P)$, and since it also acts freely on $P/\Phi(P)$, by Lemma \ref{l:fpfaction} $P$ is a special $2$-group and $\Omega_1(P)=\Phi(P)$; thus $|P|\leq 2^6$, and by Proposition \ref{p:smallorder}, $(P,\phi)$ lies in the set $\mc P$, and we easily see that $P\cong Q_{3,3}$, as claimed.
\end{proof}

Having dealt with the case where there are no $\phi$-invariant maximal subgroups, we may now assume that there is one. The following lemma reduces us to the case where $P$ is either $3$-generator or $4$-generator, and $o(\phi)=3$. Our proof uses the fact (see \cite[Theorem 5.3.2]{gorenstein}) that $\phi$ acts trivially on $Q$ if and only if $\phi$ acts trivially on $P$.

\begin{lem}\label{l:no57}
 Let $P$ be a finite $2$-group of $2$-rank $2$ with a nontrivial odd-order automorphism, $\phi$. If $P$ has a $\phi$-invariant maximal subgroup, then $P$ is $3$-generator or $4$-generator and $o(\phi) = 3$.
\end{lem}

\begin{proof}
If $P$ is $2$-generator, then the three maximal subgroups of $P$ are all stabilized by $\phi$. In particular, $\phi$ acts trivially on $P/\Phi(P)$ and hence on $P$, a contradiction. Therefore, $P$ is $3$-generator or $4$-generator. If $o(\phi) = 5$, then $P$ is 4-generator and, by Proposition \ref{p:4dimvec}, $\phi$ acts fixed-point freely on the maximal subgroups of $P$, contradicting the assumption of a $\phi$-invariant maximal subgroup. Finally, let $P$ be a $2$-group of minimal order such that there exists an automorphism of $P$ of order $7$, and let $Q$ denote a $\phi$-invariant maximal subgroup of $P$. If $\phi$ centralizes $Q$, then $\phi$ is the identity on $P$, a contradiction.  Therefore, $\phi$ induces an automorphism of order $7$ on $Q$, contradicting the minimality of $P$.
\end{proof}

\begin{lem}\label{l:RisQ8} Let $P$ be a finite $2$-group of $2$-rank $2$ with a nontrivial odd-order automorphism, $\phi$. If $P$ has a $\phi$-invariant maximal subgroup $Q$, then $[P,\phi]=[Q,\phi]$ and $P=[Q,\phi]\C_P([Q,\phi])$. Moreover, if $|P|\geq 2^8$ then $[P,\phi]$ is isomorphic to $Q_8$.
\end{lem}
\begin{proof} Set $R = [Q, \phi]$. As $\phi$ acts trivially on $P/Q$ and $Q/R$, it acts trivially on $P/R$ and hence $[P,\phi] = R$. Moreover, by 
\cite[Theorem 5.3.5]{gorenstein}, $P=R\C_P(\phi)$, and so the three subgroup lemma implies that $[R,\C_P(\phi)] = [[Q,\phi],
C_P(\phi)]=1$. Thus, $\C_P(\phi)=\C_P(R)$, proving that $P=R\C_P(R)$, as claimed. For the last statement, note the result holds when $|P| = 2^8$ by Proposition \ref{p:smallorder}. If $|P| \geq 2^9$ and $Q$ has a $\phi$-invariant maximal subgroup, then by induction on $|P|$, $[P, \phi] = [Q,\phi] \cong Q_8$. If $Q$ does not have a $\phi$-invariant maximal subgroup, then $Q \cong C_{n,n}$ by Proposition \ref{p:nophiinvmaxs}. It follows that $Q/\Omega_1(Q) \cong C_{n-1,n-1}$ and so, by induction on $n$, $\phi$ acts trivially on $P/\Omega_1(Q)$ (since $n\geq 4$); therefore, $\phi$ acts trivially on $P$ as $\Omega_1(Q) \leq \Phi(Q) \leq \Phi(P)$ (also since $n \geq 4$).
\end{proof}

\begin{prop} Let $P$ be a finite $2$-group of $2$-rank $2$ with a nontrivial odd-order automorphism, $\phi$, and a $\phi$-invariant maximal subgroup. If $P$ is $3$-generator, then $P\cong Q_8\wr C_2$, $QC_{3,n}$ or $QC^*_{3,n}$.
\end{prop}
\begin{proof} By Proposition \ref{p:smallorder}, the proposition is true for $|P|\leq 2^9$. If $|P|\geq 2^{10}$, $Q$ denotes a $\phi$-invariant maximal subgroup and $R = [Q,\phi]$, then Lemma \ref{l:RisQ8} implies that $P=R\C_P(R)$ and $R \cong Q_8$. Thus, $P$ is a quotient of $X=R\times \C_P(R)$ by a (central) subgroup $Y$ of order $2$. As $P$ is $3$-generator, $X$ is either $3$-generator or $4$-generator (so that $\C_P(R)$ is either $1$-generator or $2$-generator), and if $X$ is $4$-generator then $Y\cap\Phi(X)=1$, so that $Y$ is a direct factor of $X$. In this case, $\C_P(R) \cong C_2\times C_{2^n}$, and hence $X\cong Q_8\times C_2\times C_{2^n}$, with $Y$ the direct factor of order $2$; this yields $P\cong QC_{3,n}$. If $X$ is $3$-generator, then $\C_P(R)$ is cyclic and $X\cong Q_8\times C_{2^n}$; this implies $P\cong QC^*_{3,n}$, completing the proof.
\end{proof}

%If $Q$ is $4$-generator, we consider the group $G=P/\Phi(Q)$; since $\Phi(Q)$ is a characteristic subgroup of $Q$, $\phi(Q)$ is a $\phi$-invariant normal subgroup of $P$. The group $G$ has order $2^5$, is $3$-generator, has an elementary abelian subgroup of index $2$, and $\Aut(G)$ is not a $2$-group. These properties are enough to determine $G$ up to isomorphism using Magma, and in fact $G\cong V_4\wr C_2$ (number 27 on the SmallGroup database). If $\theta$ is an automorphism of $G$ of order $3$, then $\theta$ fixes the elementary abelian maximal subgroup $H$, and $\theta$ fixes no maximal subgroups of $H$. This implies that $\phi$ fixes no maximal subgroups of $Q$, and so $|Q|=2^6$ by Proposition \ref{p:nophiinvmaxs}. Hence $|P|=2^7$, and $P\cong Q_8\wr C_2$.

\begin{prop}\label{p:4genhas3gen} Let $P$ be a finite $2$-group of $2$-rank $2$ with a nontrivial odd-order automorphism, $\phi$. If $P$ is $4$-generator and has a $3$-generator $\phi$-invariant maximal subgroup $Q$, then $Q$ has a $\phi$-invariant maximal subgroup and:
\begin{enumerate}
\item $Q \ncong Q_8\wr C_2$;
\item if $Q\cong QC_{3,n}$, then $P\cong \GA_{n+4}$, $\GB_{n+4}$ or $Q_{3,n+1}$;
\item if $Q\cong QC_{3,n}^*$, then $P\cong QD_{3,n+1}^*$.
\end{enumerate}
\end{prop}
\begin{proof} Firstly, since $Q$ is $3$-generator and $\phi$ does not act trivially on $Q$, Proposition \ref{p:nophiinvmaxs} implies that $Q$ has a $\phi$-invariant maximal subgroup. Proposition \ref{p:smallorder} implies that we may assume $|P| \geq 2^{10}$ (in particular, this shows that $Q$ cannot be $Q_8\wr C_2$, as then $|P|=2^8$); by Lemma \ref{l:RisQ8}, we have $R=[Q,\phi]$ is isomorphic with $Q_8$ and $P=R\C_P(R)$. 

Suppose that $Q\cong QC_{3,n}$, so that $\C_Q(R)\cong C_2\times C_{2^n}$, and let $x$ be an element of $\C_P(R)\setminus \C_Q(R)$.  Obviously $x^2$ acts trivially on $R$ and $x^2\in\C_Q(R)$, an abelian group, so $x^2\in \Z(Q)$. Write $R=\gen{a,b}$, with $a$ and $b$ of order $4$, and write $S=\gen c$ for the cyclic direct factor of $Q$. It is easy to see that $S$ and $\bar S=\gen{a^2c}$ are the only two normal subgroups of $Q$ that are cyclic of order $2^n$, and so either $S^x=S$ or $S^x=\bar S$. Let $z$ be the involution in $S$, and notice that $\gen{a^2,z}$ is a central, elementary abelian subgroup of $P$ (as $x$ centralizes both $a^2$ and $z$). Hence, $P$ has three involutions. Also, $S\cap \bar S=\gen{c^2}$ is characteristic in $Q$, so $x$ sends $c^2$ to some power of itself. Let $W=\gen{a,c^{2^{n-2}}}$, a normal subgroup of $Q$ isomorphic with $C_4\times C_4$. By Theorem \ref{t:jankoCPWmetacyclic}, $\C_Q(W)$ is metacyclic and $\Omega_2(\C_Q(W))=W$. Since $|Q:\C_Q(W)|=2$, and $\C_Q(W)$ is $2$-generator and $P$ is $4$-generator, $\C_P(W)=\C_Q(W)$. Hence, $x$ does not centralize $W$, but does centralize $R$; thus, $x$ inverts $c^{2^{n-2}}$. In particular, $\Z(P)=\gen{a^2,z}$, so that $x$ has order $4$.

If $S^x=S$, then $c^x=c^\alpha$, where $\alpha \in \{1,-1,2^{n-1}-1,2^{n-1}+1\}$. The two cases $\alpha=1$ and $\alpha=2^{n-1}+1$ cannot arise as $x$ inverts $c^{2^{n-2}}$. If $\alpha=2^{n-1}-1$ and $x^2$ is one of $a^2$, $a^2z$ or $z$, then $ax$, $acx$ and $cx$, respectively, are involutions in $P$ and not in $Q$, a contradiction. If $\alpha=-1$, then $x^2=z$ gives $P\cong Q_{3,n+1}$ and $x^2=a^2z$ gives $P\cong \GB_{n+4}$. If $x^2=a^2$ then $\gen{x,a,b}$ would be a group with a single involution, as $a^2z$ and $z$ lie outside it, but this is impossible, so this situation cannot arise. 

If $S^x=\bar S$, then as above, $c^x=a^2c^\alpha$, where $\alpha\in \{1,-1,2^{n-1}-1,2^{n-1}+1\}$. As before, the two cases $\alpha=1$ and $\alpha=2^{n-1}+1$ cannot arise as $x$ inverts $c^{2^{n-2}}$. If $x^2=a^2$, $a^2z$ and $z$, then $ax$, $cx$ and $acx$, respectively, is an involution not lying in $Q$, and so $P$ has more than three involutions. Thus $c^x=a^2c^{-1}$; again, $x^2\neq a^2$, and if $x^2=z$ then $P\cong \GA_{n+4}$. If $x^2=a^2z$ then $(ax)^2=z$ and we again get $\GA_{n+4}$. This completes the proof of (1).

Finally, suppose that $Q\cong QC_{3,n}^*$. Here, $\C_Q(R)$ is cyclic of order $2^n$, so that $\C_P(R)$ contains a cyclic maximal subgroup. If $\C_P(R)$ is cyclic itself then $P$ is $3$-generator, and if $\C_P(R)$ is dihedral then $P\cong QD_{3,n+1}^*$, as claimed. It remains to show that $\C_P(R)$ cannot be $C_{2^n}\times C_{2}$, semidihedral, generalized quaternion, or a modular $2$-group, and to do this, in each case we will exhibit an elementary abelian subgroup of rank $3$.

If $\C_P(R)$ contains a subgroup, $S_1$,  isomorphic to $Q_8$, then let $S$ denote the subgroup generated by the diagonal subgroup of $R \times S_1$ and the center of $R \times S_1$, a group of order 16 whose image in $R\ast \C_P(R)=P$ is elementary abelian of order $8$, and this proves that if $\C_P(R)$ is semidihedral or generalized quaternion then $P$ has $2$-rank $3$. If $\C_P(R)$ is $C_{2^n}\times C_2$, then it is easy to see that $P$ has $2$-rank $3$ or is $3$-generator (depending on whether you identify the involution in the cyclic subgroup or one outside), and if $\C_P(R)$ is a modular $2$-group then it contains a maximal subgroup isomorphic with $C_{2^{n-1}}\times C_2$, with the central involution of $\Modul_{n+1}$ being in the larger cyclic subgroup, and so the central product of this with $R$ has $2$-rank $3$ as we have stated above. Thus, all parts are proved.
\end{proof}

\begin{prop}\label{p:4gen3genphiinv} Let $P$ be a finite $2$-group of $2$-rank $2$ with a nontrivial odd-order automorphism, $\phi$. If $P$ is $4$-generator and has a $\phi$-invariant maximal subgroup, then $P$ has a $3$-generator $\phi$-invariant maximal subgroup.
\end{prop}
\begin{proof} By Proposition \ref{p:4dimvec}, $P$ has three  $\phi$-invariant maximal subgroups: call them $Q_1$, $Q_2$ and $Q_3$ and assume each is $4$-generator. 
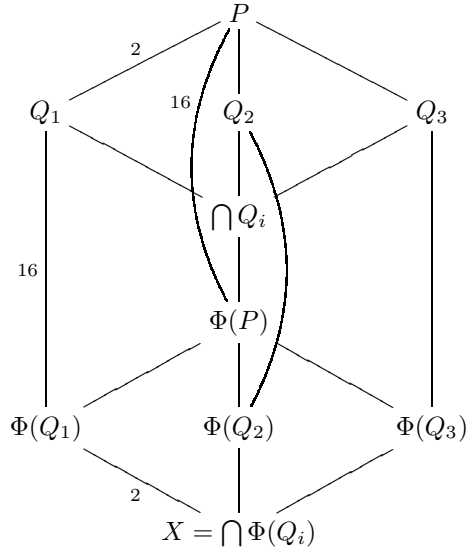
\[
\xymatrix{
& P \ar@{-}[dl]_{2}  \ar@{-}[d]  \ar@{-}[dr] \ar@/_1.5pc/@{-}[ddd]_(.3){16}  & \\
Q_1  \ar@{-}[ddd]_{16}  \ar@{-}[dr] & Q_2  \ar@{-}[d]  \ar@/^1.5pc/@{-}[ddd] & Q_3  \ar@{-}[dl]  \ar@{-}[ddd] \\
& \bigcap Q_i  \ar@{-}[d] & \\
& \Phi(P)  \ar@{-}[dl]  \ar@{-}[d]  \ar@{-}[dr] & \\
\Phi(Q_1)  \ar@{-}[dr]_{2}  & \Phi(Q_2) \ar@{-}[d] &  \Phi(Q_3) \ar@{-}[dl] \\
&X= \bigcap \Phi(Q_i) &
}
\]

Firstly, the subgroups $\Phi(Q_i)$ are all different, else $Q_i/\Phi(Q_1)$ is elementary abelian of order $16$ and $P/\Phi(Q_1)$ is therefore elementary abelian of order $2^5$ (since all elements of $P/\Phi(Q_1)$ lie in one of the $Q_i/\Phi(Q_1)$, so have order $1$ or $2$), contradicting the fact that $P$ is $4$-generator. Let $X$ be the intersection of the $\Phi(Q_i)$, and let $G=P/X$. This group has the following properties.
\begin{enumerate}
\item $G$ is $4$-generator, $|G|=2^6$, and $G$ possesses an automorphism $\phi$ of order $3$.
\item The three maximal $\phi$-invariant subgroups of $G$ are each $4$-generator, and have abelianization of order $16$.
\end{enumerate}
The last statement, that $|Q_i/Q_i'|=16$, follows by the fact that each $Q_i$ has, by induction, a $3$-generator $\phi$-invariant maximal subgroup that, by Proposition \ref{p:4genhas3gen}, is one of the groups  $\GA_n$, $\GB_n$, $Q_{3,n}$ and $QD_{3,n}^*$, and examining the presentations of these groups, we see that in each case the derived subgroup and Frattini subgroup coincide.

One may use Magma to find all such groups of order $2^6$. There are two $4$-generator groups $G$ of order $2^6$ that possess an automorphism $\phi$ of order $3$ such that the three $\phi$-invariant subgroups of $G$ are each $4$-generator, given labels 192 and 194 in the SmallGroup database. The first is abelian, so in particular the $Q_i$ are abelian, and while the second is nonabelian, the $Q_i$ are all isomorphic and abelian. Hence there is no such group $G$, and thus no such group $P$ can exist.
\end{proof}

This last proposition completes the $4$-generator case, and so there is no counterexample to Theorem \ref{t:oddorderauts} or Theorem \ref{t:oddorderauts2}.

\bibliography{references}

\end{document}